\documentclass[a4paper,11pt]{article}
\usepackage{amsfonts}
\usepackage{amssymb}
\usepackage{amsmath}
\usepackage{amsthm}
\usepackage{multirow}
\usepackage{diagbox}
\usepackage{amsopn}
\usepackage{gensymb} 
\usepackage{fancyhdr}
\usepackage{mathrsfs}
\usepackage{latexsym}
\usepackage{graphicx}
\usepackage{subfigure}
\usepackage{float} 
\usepackage[]{caption2}

\renewcommand{\thesubfigure}{(\roman{subfigure})}
\makeatletter \renewcommand{\@thesubfigure}{\thesubfigure \space}
\renewcommand{\p@subfigure}{} \makeatother

\usepackage{color}
 \usepackage{calc}
  \usepackage{tikz}
  
\usepackage{setspace}
\usepackage{textcomp} 

\usepackage[width=11cm]{geometry}   
\usepackage{varwidth}
\usepackage{algorithm}
\usepackage{multicol}
\usepackage{algpseudocode}
\numberwithin{equation}{section} \allowdisplaybreaks
\theoremstyle{plain}
\newtheorem{thm}{Theorem}[section]
 \usepackage{rotating}
\usepackage{appendix}

\theoremstyle{remark}
\newtheorem{remark}{Remark}
\numberwithin{equation}{section}
\newtheorem{Lemma}[thm]{Lemma}
\newtheorem{prop}[thm]{Proposition}

\begin{document}
\title{\bf{A parametric Bayesian level set approach for acoustic source identification using multiple frequency information}}

\author{\hspace{-2.5cm}{\small Zhi-Liang Deng$^{1}$, Xiao-Mei Yang$^2$ \thanks{
Corresponding author: yangxiaomath@163.com; yangxiaomath@home.swjtu.edu.cn
Supported by NSFC No. 11601067, 11771068 and No.11501087, the Fundamental Research Funds for the Central Universities No. 2682018ZT25 and ZYGX2018J085.
} and Jiangfeng Huang$^{1}$} \\
\hspace{-1.5cm}{\scriptsize $1.$ School of Mathematical Sciences,  University of Electronic Science and Technology of China,
Chengdu 610054, China}\\
\hspace{-2.5cm}{\scriptsize $2.$ School of Mathematics,
Southwest Jiaotong University
Chengdu 610031, China}
}
\date{}
\maketitle

\begin{abstract}
\noindent The reconstruction of the unknown acoustic source is studied using the noisy multiple frequency data on a remote closed surface. Assume that the unknown source is coded in a spatial dependent piecewise constant function, whose support set is the target to be determined. In this setting, the unknown source can be formalized by a level set function. The function is explored with Bayesian level set approach.  
To reduce the infinite dimensional problem to finite dimension, we parameterize the level set function by the radial basis expansion.  
The well-posedness of the posterior distribution is proven. The posterior samples are generated according to the Metropolis-Hastings algorithm and  the sample mean is used to approximate the unknown.
Several shapes are tested to verify the effectiveness of the proposed algorithm. These  numerical results show that the proposed algorithm is feasible and competitive with the Mat\'ern random field for the acoustic source problem.



\noindent \textbf{Key words:} Level set; Bayesian inversion; Acoustic source; 
Radial basis; Mat\'ern random field prior

\noindent \textbf{MSC 2010}: 35R20, 65R20
\end{abstract}

\section{Introduction}
Sound source detection has been the active research topics in the field of acoustic scattering problems due to the extensive applications to various fields such as military surveillance to detect flying objects or armored vehicles and a humanoid robot auditory system. A main object to be determined for this kind of problems is the location and shape of the sound source. These information are coded in a spatial dependent function with compact support.  For reconstructing this function, the multiple frequency sound wave data in a remote close surface are usually collected. The observed data is related to the unknown function by a partial differential equation. 
The task of such source discovery is an inverse problem. It is usually ill-posed, i.e.,  the solution is highly sensitive to changes in the observed data. 
A standard way of solving inverse problems for seeking the source function is to minimize some functional related with the collected data and simulated data computed with a candidate source function. Some iterative procedure is typically applied to achieve the minimizer of the functional. To cope with the numerical instability or make the computation more stable, a regularized term is imposed on the minimization function, named the regularization method.

In this paper, we consider using the Bayesian level set technique to reconstruct the unknown acoustic source function with compact support. The level set approach is proposed in \cite{osher1, osher2} to track the wave front interface.  Compared with the classical curve parameterization method, the level set approach allows for topological changes to be detected during the course of algorithms.  This advantage enables it to become the popular technique to deal with interface problems. 
In inverse problems, it has also been discussed extensively \cite{burger, dorn, ito, santosa1, tai}. It can be seen that these studies cope with the shape reconstruction problem in the classical framework, rarely analyze the uncertainty of the unknown variables. The uncertainty is often characterized by some statistical techniques. And moreover the characterization provides the estimation for the unknown variables and their reliability information. To quantify the uncertainty, the Bayesian level set algorithm is studied in \cite{dunlop_1, iglesias1}.  One can refer to \cite{lorentzen1, lorentzen2, ping, tai} for some related applications of this algorithm. 
In Bayesian level set inversion, the reconstruction target is the posterior distribution, which provides us a basis for quantifying the uncertainty.
 In this process, the prior distribution needs to be treated firstly, actually before the data is acquired. 
 We use the Whittle-Mat\'ern random field as the prior information of the level set function. Such random field has extensive applications and is studied by many authors \cite{lindgren, rasmussen, roininen}. In \cite{chada, dunlop_1}, it is used to characterize the prior of the level set function. As seen in these references, the simulation of the random field can be obtained by solving some related stochastic differential equation or utilizing the Karhunen-Lo\`eve expansion by the eigen-system of the covariance function. The stochastic differential equation can be solved by the finite difference or the finite finite approach \cite{lindgren, rasmussen, roininen}. In regular domain, the eigen-system can be in general obtained directly \cite{chada, dunlop_1}. However, in non-regular domain, it is still necessary to use some numerical approximation algorithms to obtain the approximated eigen-system \cite{jiang}.  It can be seen that the prior discretization, whatever by the stochastic differential equation or the Karhunen-Lo\`eve expansion, will induce a high dimensional unknown variable. This may cause huge computation burden in Bayesian inversion process. To reduce the computation complexity, some acceleration algorithms have been explored for  statistical inversion. It is well-known that the surrogate model algorithms can reduce the numerical complexity effectively \cite{deng}.  The surrogate model methods try to use the polynomial chaos expansion to replace the forward model or the likelihood function. And in doing so, we just need to evaluate the polynomials instead of the forward model or likelihood function, which significantly reduces the numerical burden in Bayesian inversion. 
 
In our problem, we focus on using the radial basis function (RBF) expansion to parameterize the level set function. The parametric level set approach is considered  in the classical frame \cite{aghasi, liud, pingen}. To the authors' knowledge, it still has not been studied in statistical inversion. We compare this numerical effects for  using the Whittle-Mat\'ern random field prior with the RBF expansion prior. And the well-posdeness of the posterior distribution based on the RBF expansion prior is explained.

The remainder is organized in the following: In Section 2, we depict the inverse acoustic wave source reconstruction problem. In Section 3, the Bayesian level set approach based on the Whittle-Mat\'ern random field prior is discussed. 
In Section 4, we propose the RBF parametric level set algorithm to solve the reconstruction problem. We give some numerical tests to verify the proposed algorithms in Section 5.




\section{Acoustic source problem}
The pressure $u$ of time-harmonic wave radiated from source with density $A(k)f(x)$ is modeled by \cite{eller}
\begin{align}
\label{in1.1}
&(\triangle+k^2)u(x, k)=A(k)f(x), x\in \mathbb{R}^d,\\
&\lim_{r\rightarrow \infty}r^{\frac{d-1}{2}}{\big\{}\frac{\partial u}{\partial r}-iku{\big\}}=0, \,\, r=|x|,\label{in1.2}
\end{align}
where $k=\omega/c_0$ is the wave number, $\omega$ is the radial frequency,  $c_0$ is the speed of sound, $d$ is the spatial dimension and \eqref{in1.2} is the Sommerfeld radiation condition, $A(k)$ is a given function and $f$ is the unknown with compact support $D\subset\mathbb{R}^d$ to be determined. 
The multiple frequency observed data for the acoustic wave field is collected in a remote closed surface $\partial\Omega$. Suppose the acoustic source domain $D$ is compactly contained in  the region $\Omega$  bounded by the closed surface $\partial\Omega$. 
For simplicity, we consider 2-dimensional case, i.e., $d=2$ and let $A(k)\equiv 1$.  
According to the potential representation, we have  \cite{eller}
\begin{align}
u(k, x)&=\int_D f(y)H_0^{(1)}(k|x-y|)dy, \,\, x\in\partial\Omega, \label{in1.4}
\end{align}
where $H_0^{(1)}$ is the cylindrical Hankel function. The data are collected on finite points $\vec{x}:=\{x_1, x_2, \cdots, x_N\}\subset\partial\Omega$ for finite wave numbers $\{k_1, k_2, \cdots, k_M\}\subset [k_{\min}, k_{\max}]$.  
Assume that the noise is additive Gaussian,  $\eta\sim N(0, \delta^2 I)$, $\delta$ is the noise mean variance. By \eqref{in1.4}, the system can be written as
\begin{align}\label{in1.5}
\mathcal{H}_{k_m}f+\eta_m=[u(x_j, k_m)]_{j=1}^N:=b_m \,\, \text{for}\,\, m=1, 2, \cdots M,
\end{align}
where $\mathcal{H}_{k}: L^2(D)\rightarrow \mathbb{C}^{N}$  depends on wave number $k$ and $\eta_m$ is the noise. Denote $\mathcal{H}:=[\mathcal{H}_{k}]_{k=k_1}^{k_M}$ and $b=[b_1; b_2; \cdots; b_M]$. In this paper, we assume that the function $f$ is known a priori to have the form
\begin{align}
\label{lev2.1}
f(x)=\sum_{l=1}^n w_l \mathbb{I}_{D_l}(x);
\end{align}
here $\mathbb{I}_D$ denotes the indicator function of subset $D\subset\mathbb{R}^2$, $\{D_l\}_{l=1}^n$ are subsets of $D$ such that
$\bigcup_{l=1}^n\bar{D}_l=\bar{D}$ and $D_l\bigcap D_j=\varnothing$ ($l\neq j$), the $\{w_l\}_{l=1}^n$ are known positive constants. In this setting, the regions $D_l$ determine the unknown function and therefore become the primary unknowns. 
Some papers have discussed the acoustic source reconstruction problems \cite{alzaalig,bao,eller,wang_2}. From these existed results, one can see that the domain $D$ can be determined uniquely by the multiple frequency data.

\section{Level set Bayesian inversion with Whittle-Mat\'ern random field prior}

As stated above, instead of reconstructing the acoustic source, we solve the recovery problem of interfaces $\partial D_l$ or regions $D_l$. It can be devised as an interface tracking problem. A simple and versatile approach for this kind of problems is the level set technique. In this method, the interface $\partial D_l$ is characterized by a so-called level set function. Actually, the interface $\partial D_l$ is represented as the contour curve at some fixed value of the level set function, which is a higher dimensional function.
For our problem, the level set function $\varphi$ is defined by the following way \cite{dorn,dunlop_1,iglesias1}:
\begin{align}\label{lev2.2}
D_l=\{x\in D\mid c_{l-1}\leq\varphi(x)<c_{l}\}\,\,\text{for}\,\, l=1, 2, \cdots, n,
\end{align}
where $-\infty=c_0<c_1<\cdots<c_n=\infty$ are constants.
By this characterization, the source function is determined. Define the level set map
\begin{align}
\label{lev2.3}
\mathcal{G}: C(\bar{D})\rightarrow L^2(D),\,\,
\mathcal{G}(\varphi)=f.
\end{align}
Since there exist many different $\varphi$ determine the same $f$, the level set map is not an injection. However, if the function $\varphi$ is fixed, the function $f$ is uniquely specified. 
With \eqref{in1.5} and \eqref{lev2.3}, we have the following operator equation
\begin{align}
\label{lev2.4}
\mathcal{K}(\varphi):=\mathcal{H}\circ\mathcal{G}(\varphi)=b.
\end{align}
We deal with the equation in the frame of Bayesian inversion. Bayesian perspective views the involved quantities $\varphi, b, \eta$ as random variables. Denote the random variable $\varphi|b$,  $\varphi$ given $b$. The solution of Bayesian reconstruction is the posterior probability distribution $\mu(\varphi|b)$  that expresses our beliefs regarding how likely the different parameter values are. 
The posterior density is denoted by $\pi(\varphi|b)$. Bayes' formula shows that
\begin{align}
\label{lev2.5}
\pi(\varphi|b)=\frac{\pi(b|\varphi)\pi(\varphi)}{\pi(b)}\propto \pi(b|\varphi)\pi(\varphi),
\end{align}
where $\pi(b|\varphi)$ is the likelihood function, $\pi(\varphi)$ the prior density and $\pi(b)$ the constant. Since the noise is additive Gaussian $\eta\sim N(0, \delta^2I)$, the likelihood function is given by
\begin{align}
\label{lev2.6}
\pi(b|\varphi)\propto\exp(-\frac{|\mathcal{K}(\varphi)-b|^2}{2\delta^2}):=\exp(-\Phi(\varphi)).
\end{align}

 A frequently used prior for the unknown $\varphi$ is the Whittle-Mat\'ern random field $N(0, C)$. The covariance function is given by
\begin{align}
\label{pr4.1}
C(x)=\frac{2^{1-\nu}}{\Gamma(\nu)}(\frac{|x|}{l})^\nu K_\nu(\frac{|x|}{l}), \,\, x\in\mathbb{R}^d,
\end{align}
where $\nu>0$ is the smoothness parameter, $K_\nu$ is the modified Bessel function of the second kind of order $\nu$ and $l$ is the length-scale parameter. The integer value of $\nu$ determines the mean square differentiability of the underlying process, which matters for predictions made using such a model. This distribution is discussed by many authors \cite{lindgren, rasmussen, roininen}.  A method for explicit, and computationally efficient, continuous Markov representations of Gaussian Mat\'ern fields is derived  \cite{lindgren}.
The method is based on the fact that a Gaussian Mat\'ern field on $\mathbb{R}^d$
can be viewed as a solution to the stochastic partial differential equation (SPDE)
\begin{align}\label{pr4.2}
(I-l^2\triangle)^{(\nu+d/2)/2}\varphi=\sqrt{\alpha l^2}W,
\end{align}
where $W$ is the Gaussian white noise and the constant $\alpha$ is 
\begin{align*}
\alpha:=\sigma^2\frac{2^d\pi^{d/2}\Gamma(\nu+d/2)}{\Gamma(\nu)}.
\end{align*}
Here $\sigma^2$ is the variance of the stationary field.
The operator $(I-l^2\triangle)^{(\nu+1)/2}$ is a pseudo-differential operator defined by its Fourier transform. When $\nu\in\mathbb{Z}$, we can solve \eqref{pr4.2} by the finite element method with suitable boundary condition. The stochastic weak solution of the SPDE \eqref{pr4.2} is found by
requiring that
\begin{align}\label{mart3.8}
(\psi, (I-l^2\triangle)^{(\nu+d/2)/2}\varphi)=(\psi, \sqrt{\alpha l^2}W),
\end{align}
where $\psi\in L^2(D_1)$. We use the linear element to solve the weak formulation with Dirichlet boundary condition, i.e., $\varphi\mid_{\partial D_1}=0$. 


For Bayesian inversion, the well-posedness of the posterior distribution $\pi(\varphi|b)$ can be established by some assumptions on the forward operator and the prior distribution on $\varphi$ \cite{iglesias1, stuart}. It can be seen from \cite{iglesias1} that this well-posedness requires the continuity of the level set map.
 In \cite{iglesias1}, this point has been discussed and is given in the following proposition
\begin{prop}\cite{iglesias1}\label{iglesiaspro1}
For $\varphi\in C(\bar{D})$ and $1\leq q<\infty$, the level set map $\mathcal{G}: C(\bar{D})\rightarrow L^q(D)$ is continuous at $\varphi$ if and only if $\mathfrak{m}(\bar{D}_l\cap \bar{D}_{l+1})=0$ for all $l=1, 2, \cdots, n-1$. Here $\mathfrak{m}$ denotes the Lebesgue measure and $\bar{D}_l\cap \bar{D}_{l+1}$ is the $l$-th level set $\{x\in D| \varphi(x)=c_l\}$.
\end{prop}

\section{Parametric level set based on RBF expansion}

Bayesian inversion requires the numerical evaluation of the forward problem repeatedly for a large number of samples. An appropriate parameterization of the unknown function can significantly reduce the dimensionality of an inverse problem.  A low order model based on the radial basis expansion for the level set function is discussed in \cite{aghasi, liud, pingen}, where the level set function is represented by the radial basis function expansion. 
This representation provides flexibility in terms of its ability to characterize shapes of varying degrees with varying degrees of complexity as measured specifically by the curvature of the boundary \cite{aghasi}.  
 We represent the unknown level set function using the radial basis expansion 
\begin{align}\label{rbf4.1}
\varphi(x, a)=\sum_{s=1}^J a_s p(x, \theta_s),
\end{align}
where $a_s$ are the random coefficients, $p(x, \theta_s)$ the radial basis function and $\theta_s$ called the RBF centers. 
Here, we take the Gaussian radial basis function
\begin{align}\label{rbf4.2}
p(x, \theta_s; \lambda_s)=\exp(-\frac{|x-\theta_s|^2}{2\lambda_s^2}),
\end{align} 
where $\lambda_s$ are  the Gaussian widths at centers $\theta_s$.  The kernel is Mercer kernel and therefore has the following representation \cite{santin}
\begin{align}
\label{merck}
p(x, \theta_s; \lambda_s)=\sum_{\tau=1}^{\infty} \zeta_\tau \phi_\tau(x) \phi_\tau(\theta_s),
\end{align}
where $\zeta_n$ are the eigenvalues and $\phi_n$ are the eigenfunctions of the compact
and self-adjoint integral operator $T: L^2(D)\rightarrow L^2(D)$ defined by
\begin{align}
\label{merck1}
Tf(x)=\int_D p(x, \theta; \lambda)f(\theta)d\theta.
\end{align}
Moreover, the level set function $\varphi$ can be written as
\begin{align}
\label{merck2}
\varphi(x, a)&=\sum_{s=1}^J a_s \sum_{\tau=1}^{\infty} \zeta_\tau \phi_\tau(x) \phi_\tau(\theta_s)\nonumber\\
&=\sum_{\tau=1}^{\infty} \zeta_\tau[\sum_{s=1}^J a_s\phi_\tau(\theta_s)]   \phi_\tau(x).
\end{align}
Denote $a:=[a_1; a_2; \cdots; a_J]$.  We assume that the components $a_s$ $(s=1, 2, \cdots, J)$ are independent and identically distributed random variables. Denote the probability space by $(\mathbb{R}^J, \mathfrak{B}_J, \mu_a)$.
Define the map according to \eqref{rbf4.1}
\begin{align}
\label{rbf4.2_1}
  \begin{split}
  & \mathcal{T}:(\mathbb{R}^J, \mathfrak{B}_J, \mu_a)\rightarrow (C(\bar{D}), \Sigma, \mu_0), \\
  &\mathcal{T}(a)=\varphi. 
    \end{split}
\end{align}
The measure space $(C(\bar{D}), \Sigma, \mu_0)$ is the push-forward of $(\mathbb{R}^J, \mathfrak{B}_J, \mu_a)$ under $\mathcal{T}$. The measure $\mu_0$ is well-defined since the RBF is continuous. Actually, the new measure $\mu_0$ is defined by 
\begin{align}
\mu_0(A)=\mu_a(\mathcal{T}^{-1}(A))\,\, \text{for}\,\, A\in \Sigma. 
\end{align}

With the expansion \eqref{rbf4.1}, we have the recovery problem of the coefficient $a$
\begin{align}
\label{rbf4.4}
\mathcal{L}:\mathbb{R}^J\rightarrow \mathbb{C}^{NM},\,\,\mathcal{L}(a):=\mathcal{K}\circ\mathcal{T}(a)=\mathcal{H}\circ\mathcal{G}\circ\mathcal{T}(a)=b.
\end{align}
This likelihood function is given for \eqref{rbf4.4}
\begin{align}
\label{rbf4.5}
\pi(b|a)\propto\exp(-\frac{|\mathcal{L}(a)-b|^2}{2\delta^2}):=\exp(-\Psi(a; b)).
\end{align}
Therefore, the posterior density satisfies
\begin{align}
\label{rbfp4.6}
\pi(a|b)\propto \exp(-\Psi(a; b))\pi_a(a),
\end{align}
where $\pi_a$ is the prior density for $a$. 
\begin{thm}
If $a$ is $J$-dimensional Gaussian, i.e., $a\sim N(0, I)$, then $\mu_0$ is the Gaussian measure $N(0, \mathcal{C})$. The covariance function of this random field is of the form
\begin{align}\label{rbf4.6}
C(x, \tilde{x})=\sum_{s=1}^J \exp(-\frac{|x-\theta_s|^2+|\tilde{x}-\theta_s|^2}{2\lambda_s^2}).
\end{align}
\end{thm}
\begin{proof}
 We can get the covariance function of the random field 
\begin{align*}
&C(x, \tilde{x})=\text{cov}(x, \tilde{x})=\text{cov}(\varphi(x), \varphi(\tilde{x}))\nonumber \\
&\nonumber=[p(x, \theta_1; \lambda_1), p(x, \theta_2; \lambda_s), \cdots, p(x, \theta_J; \lambda_J)]
\left[\begin{array}{c}p(\tilde{x}, \theta_1; \lambda_1) \\p(\tilde{x}, \theta_2; \lambda_2) \\ \vdots\\p(\tilde{x}, \theta_J; \lambda_J)\end{array}\right]\\
&\nonumber=\sum_{s=1}^J p(x, \theta_s; \lambda_s)p(\tilde{x}, \theta_s; \lambda_s)=\sum_{s=1}^J p(x, \theta_s; \lambda_s)p(\tilde{x}, \theta_s; \lambda_s)\\
&=\sum_{s=1}^J \exp(-\frac{|x-\theta_s|^2+|\tilde{x}-\theta_s|^2}{2\lambda_s^2}).
\end{align*}
\end{proof}
By virtue of the parallelogram law, it follows that
\begin{align}
\text{cov}(x, \tilde{x})\leq J\exp(-\frac{|x-\tilde{x}|^2}{4\lambda^2}),
\end{align}
where $\lambda=\max_{s}\lambda_s$. 

Next, we discuss that the RBF expansion of the level set function can meet the condition in Proposition \ref{iglesiaspro1}, i.e., the measure of the level set vanishes almost surely on both $\mu_a$ and $\mu_0$.
Define the functional $\mathfrak{M}_c: C(\bar{D})\rightarrow \mathbb{R}^+\cup \{0\}$ by
$$\mathfrak{M}_c(\varphi):=\mathfrak{m}(\{x\in D| \varphi(x)=c\}).$$ 
\begin{Lemma}\label{lemm4.1}
For any $c\in\mathbb{R}$, $\mathfrak{M}_c$ is $\mu_0$-measurable.
\end{Lemma}
\begin{proof}
It suffices to verify for any $t\in\mathbb{R}$, the set $M_t:=\{\varphi\in C(\bar{D})| \mathfrak{M}_c(\varphi)<t\}\in \Sigma$.  Since $\mathfrak{M}_c$ is non-negative, for $t\leq 0$, we know that $M_t=\varnothing$ and hence measurable. For $t>0$, it can be obtained that $M_t$ is open. In fact, if it is not open, then there exists $\varphi_0\in M_t$, such that $\varphi_0$ is not an interior point. That means that for every $n\in\mathbb{N}$, there exists $\varphi_n$, even if it holds that $\|\varphi_n-\varphi_0\|_\infty<\frac{1}{n}$, we still have $\varphi_n\notin M_t$. Therefore we have $\mathfrak{m}(\{x|\varphi_n(x)=c\})\geq t$. Moreover, by the construction of $\varphi_n$ and the triangle inequality, we have $\{x|\varphi_n(x)=c\}\subset B_n:=\{x|\|\varphi_0(x)-c\|_\infty <\frac{1}{n}\}$. Noting that 
\begin{align*}
\{x|\varphi_0(x)=c\}=\bigcap_{n=1}^\infty B_n.
\end{align*}
and that $B_n$ is decreasing, we can conclude that
\begin{align*}
t\leq \lim_{n\rightarrow\infty}\mathfrak{m}(\{x|\varphi_n(x)=c\})\leq \lim_{n\rightarrow\infty}\mathfrak{m}(B_n)=\mathfrak{m}(\{x|\varphi_0(x)=c\})<t.
\end{align*}
This is a contradiction. So $M_t$ is open for $t>0$. Then we know that $\mathfrak{M}_c$ is $\mu_0$-measurable.
\end{proof}
Introduce the random level set 
\begin{align}\label{levaa1}
D_c=D_c(\varphi(\cdot, a))=D_c(a):=\{x|\varphi(x, a)=c\}.
\end{align}
By Lemma \ref{lemm4.1}, it follows that $\mathfrak{m}(D_c)$ is a random variable on both $(\mathbb{R}^J, \mathfrak{B}_J, \mu_a)$ and $(C(\bar{D}), \Sigma, \mu_0)$. Next we demonstrate that $\mathfrak{m}(D_c)$ vanishes almost surely on these two measure spaces.
\begin{thm}
Let the coefficient vector $a$ in \eqref{rbf4.1} be drawn from the $J$-dimensional Gaussian distribution. The level set function is given by \eqref{rbf4.1}. The random level set $D_c$ of $\varphi$ is defined by \eqref{levaa1}. Then 
\begin{enumerate}
\item $\mathfrak{m}(D_c)=0$ $\mu_a$-almost surely;
\item  $\mathfrak{m}(D_c)=0$ $\mu_0$-almost surely.
\end{enumerate}
\end{thm}
\begin{proof}
For fixed $x\in D$, $\varphi(x)$ acts as a bounded linear functional on $C(\bar{D})$ and therefore $\varphi(x, \cdot)$ is a real valued Gaussian random variable, which implies
$\mu_a(\{a | \varphi(x, a)=c\})=0$ \cite{iglesias1}. 
Moreover, it follows that
\begin{align}
\label{leva1}&\mathbb{E}[\mathfrak{m}(D_c)]=\int_{\mathbb{R}^J} \mathfrak{m}(D_c)d\mu_a=\int_{\mathbb{R}^J}\int_{\mathbb{R}^d}\mathbb{I}_{\{x|\varphi(x, \omega)=c\}}dxd\mathbb{P}(\omega)\\
&=\int_{\mathbb{R}^J}\int_{\mathbb{R}^d}\mathbb{I}_{\{(x, a)|\varphi(x, a)=c\}}dxd\mu_a=\int_{\mathbb{R}^d}\int_{\mathbb{R}^J}\mathbb{I}_{\{(x, a)|\varphi(x, a)=c\}}d\mu_adx\nonumber\\
&=\int_{\mathbb{R}^d}\int_{\mathbb{R}^J}\mathbb{I}_{\{a|\varphi(x, a)=c\}}d\mu_adx=\int_{\mathbb{R}^d} \mu_a(\{a|\varphi(x, a)=c\})dx=0.\nonumber
\end{align}
Since $\mathfrak{m}(D_c)\geq 0$, it follows that $\mathfrak{m}(D_c)=0$, $\mu_a$-almost surely. Next, we show that $\mathfrak{m}(\{x\in D| \varphi(x, \omega)=c\})=0$  $\mu_0$-almost surely. Set
\begin{align*}
M&:=\{\varphi|\mathfrak{m}(\{x| \varphi(x)=c\})=0\}\\
&=\cap_{k=1}^\infty \{\varphi|\mathfrak{m}(\{x| \varphi(x)=c\})<\frac{1}{k}\}:=\cap_{k=1}^\infty M_{\frac{1}{k}}.
\end{align*}
Since $M_{\frac{1}{k}}$ is open for any $k>0$, the set $M$ is Borel set and measurable. 
Therefore, we have
\begin{align*}
&\mu_0(M)=\mu_0 (\{\varphi|\mathfrak{m}(\{x|\varphi(x)=c\})=0\})\\
&=\mu_a(\{a|\mathfrak{m}(\{x|\sum_{s=1}^J a_s p(x, \theta_s; \lambda_s)=c\})=0\})=1.
\end{align*}
The above equality shows that  $\mathfrak{m}(\{x\in D| \varphi(x, \omega)=c\})=0$ $\mu_0$-almost surely. 
\end{proof}
\begin{remark}
According to Proposition \ref{iglesiaspro1}, the level set map $\mathcal{G}: C(\bar{D})\rightarrow L^q(D)$ is continuous at the radial basis expansion $\varphi$. With this point, we can obtain the well-posedness of the posterior distribution from the theories in \cite{iglesias1}. 
\end{remark}
The well-posedness of the posterior distribution is characterized in the sense of the Hellinger metric. Recall that the Hellinger distance between 􏱄$\mu(a|b)$ and 􏱄$\mu(a|{b^\prime})$ is defined as
\begin{align*}
d_{\rm H}\left(\mu(a|b), \mu(a|{b^\prime})\right)=\left(\frac{1}{2} \int_{\mathbb{R}^J}\left(\sqrt{\frac{d \mu(a|b)}{d \mu_a}}-\sqrt{\frac{d \mu(a|{b^\prime})}{d\mu_a}}\right)^{2} d\mu_a\right)^{\frac{1}{2}}
\end{align*}
for any measure $\mu_a$􏱅 with respect to which $\mu(a|b)$􏱄 and $\mu(a|{b^\prime})$ are absolutely continuous.
\begin{thm}
Denote $\Gamma=\delta^2 I$. The prior for $a$ is the $J$-dimensional Gaussian.
The posterior distribution $\mu(a|b)$ is locally Lipschitz with respect to $b$ in the Hellinger distance: for all $b$, $b^\prime$ with $\max\{|b|_\Gamma, |b^\prime|_\Gamma\}<r$, there exists a constant $C=C(r)>0$ such that
\begin{align}\label{concl1}
d_{\rm H}(\mu(a|b), \mu(a|{b^\prime}))\leq C|b-b^\prime|_\Gamma.
\end{align}
\end{thm}
\begin{proof}
Denote 
\begin{align*}
Z=\int_{\mathbb{R}^J} \exp(-\Psi(a; b))d\mu_a, Z^\prime=\int_{\mathbb{R}^J} \exp(-\Psi(a; b^\prime))d\mu_a.
\end{align*}
It is obvious that $1\geq Z>0$ and $1\geq Z'>0$. In fact, we can conclude that $Z, Z'$ are strictly positive. For instance, for $Z$, we have that
\begin{align*}
Z\geq \int_{|a|\leq 1} \exp(-(|b|_\Gamma+|\mathcal{K}(a)|_\Gamma))d\mu_a.
\end{align*}
By $\mathcal{K}=\mathcal{H}\circ\mathcal{G}\circ\mathcal{T}$, we know that
\begin{align*}
|\mathcal{K}(a)|_\Gamma\leq M |a|.
\end{align*}
Therefore, it follows that
\begin{align}
Z&\geq \int_{|a|\leq 1}\exp(-(|b|_\Gamma+M|a|))d\mu_a\nonumber\\
&\geq \exp(-(|b|_\Gamma+M))\mu_a(|a|\leq 1):=M_1.\label{positi}
\end{align}
Since $\mu_a$ is $J$-dimensional Gaussian, the unit ball has strictly positive measure, i.e., $\mu_a(|a|\leq 1)>0$. This shows that the posterior distribution is well-defined probability distribution.
By the definition of Hellinger distance, we have that
\begin{align*}
&2(d_{\rm H}(\mu(a|b), \mu(a|{b^\prime})))^2=\int_{\mathbb{R}^J}\left(\sqrt{\frac{d \mu(a|b)}{d \mu_a}}-\sqrt{\frac{d \mu(a|{b^\prime})}{d\mu_a}}\right)^{2} d\mu_a\\
&=\int_{\mathbb{R}^J} (\frac{\exp(-\frac{\Psi(a; b)}{2})}{Z^{\frac{1}{2}}}-\frac{\exp(-\frac{\Psi(a; b')}{2})}{Z'^{\frac{1}{2}}})^2d\mu_a\leq I_1+I_2,
\end{align*}
where 
\begin{align*}
&I_1:= \frac{2}{Z} \int_{\mathbb{R}^J} (\exp(-\frac{\Psi(a; b)}{2})-\exp(-\frac{\Psi(a; b')}{2}))^2d\mu_a,\\
& I_2:=2|Z^{-1 / 2}-(Z')^{-1 / 2}|^{2}\int_{\mathbb{R}^J} \exp(-\Psi(a; b')) d\mu_a.
\end{align*}
For $I_1$, by \eqref{positi} we have
\begin{align}\label{I1es}
I_1\leq \frac{2}{M_1}|b-b'|_\Gamma^2.
\end{align}
For $I_2$, it follows that
\begin{align}
 &\left|Z^{-1 / 2}-(Z')^{-1 / 2}\right|^{2}\nonumber \\ 
\leqslant & M \max \left(Z^{-3}, (Z')^{-3}\right)\left|Z-Z'\right|^{2} \label{I2es1}\\ 
\leqslant & M\left|b-b'\right|^{2}. \nonumber
\end{align}
Combining \eqref{I1es} and \eqref{I2es1}, we get the estimation \eqref{concl1}.
\end{proof}

\section{Numerical test}
We characterize the posterior distribution by means of sampling with the preconditioned Crank-Nicolson (pCN) MCMC method developed in \cite{stuart}. In concrete, this algorithm is implemented in the following iteration process:
\begin{center}
{\tiny{
\begin{tabular}{l|l|l}
For $i=0$, initialize &  $\varphi^{(0)}$ &  $a^{(0)}$\\
\hline
\multirow{2}{*}{For $i\geq 0$, propose} & \tiny{$\hat{\varphi}^{(i)}=\sqrt{1-\beta^2}\varphi^{(i)}+\beta \xi^{(i)}$},&\tiny{$\hat{a}^{(i)}=\sqrt{1-\beta^2}a^{(i)}+\beta \xi^{(i)}$}, \\
& where $\xi^{(i)}\sim N(0, C)$. & where $\xi^{(i)}\sim N(0, I)$.\\
\hline
Compute probability $\alpha$ & $\alpha=\min\{1, \frac{\exp(-\Phi(\hat{\varphi}^{(i)}))}{\exp(-\Phi(\varphi^{(i)}))}\}$&$\alpha=\min\{1, \frac{\exp(-\Psi(\hat{a}^{(i)}))}{\exp(-\Psi(a^{(i)}))}\}$\\
\hline
\multirow{2}{*}{Set} &$\varphi^{(i+1)}=\hat{\varphi}^{(i)}$, if $\alpha>$rand,  & $a^{(i+1)}=\hat{a}^{(i+1)}$, if $\alpha>$rand,\\
&$\varphi^{(i+1)}=\varphi^{(i)}$, otherwise. & $a^{(i+1)}=a^{(i)}$, otherwise.\\
\hline
Return to the second line.
\end{tabular}}}
\end{center}


In the numerical tests, we consider some parameter settings for the model problem given in \cite{eller}. Assume that the frequencies $k$ vary from $k_{\min}=50$ Hz to $k_{\max}=10$ kHz and $c_0=343$ms$^{-1}$.   We suppose that the unknown region $D$ is contained in a disk domain $D_1$ and the data receiver is located on $\partial\Omega$. The problem geometry is displayed in Fig. \ref{fig:digit}. The disk domain $D_1$ is partitioned into triangle element. The forward operator is evaluated by using the Gaussian integral in each triangle element in the domain $D_1$. The Whittle-Mat\'ern random field is generated by \eqref{mart3.8} in the same mesh. The radial basis centers $\theta_s$ are taken as the vertices of the triangle elements. In addition, we also use $25$ center points in the RBF expansion. The centers are displayed in Fig. \ref{centersrbf}. We first show several samples drawn according to the Whittle-Mat\'ern random field and the RBF expansion in Fig. \ref{prior1} with $\nu=1$ and $\lambda_s=\frac{\sqrt{2}}{2}$ $(s=1, 2, \cdots, J)$. The first and third rows in Fig. \ref{prior1} display the samples for the level set function $\varphi$ and the second and fourth rows give the following function
\begin{align*}
f(x)=\left\{
\begin{aligned}
1, & &\varphi(x)\geq 0, \\
0, & &\varphi(x)<0.
\end{aligned}
\right.
\end{align*} 

We display some numerical reconstructions for the case of the source with single medium in Fig. \ref{reconstruction1}. The noise level $\delta=0.01$. The prior parameters are taken as in Fig. \ref{prior1}. 
 In the MH algorithm, parameter $\beta$ is taken as $0.01$.  
In these numerical examples, we generate $2\times 10^5$ samples from the posterior distribution. 
 Some samples in the burn-in process being excluded,  the sample means are used as the approximations of the true domains.  The first two rows give the single domain case and the last  row displays two separate domains. 
 
 In Fig. \ref{reconstruction2}, we display the numerical reconstructions for the sources with two different mediums for $\delta=0.01$. In these examples, $3\times 10^5$ samples are drawn.

We use $25$ center points in the RBF expansion and show the numerical reconstructions in Fig. \ref{reconstruction3} with $3\times 10^5$ iteration steps and $\delta=0.02$.

These numerical reconstructions show that the effectiveness of the MCMC. From Fig. \ref{reconstruction2}, we see that the proposed approach can cope well with  some complex domains, which provides us conveniences in many cases. When the center points are reduced to $25$, Fig. \ref{reconstruction3} shows the RBF expansion also works well. This means that by the RBF parameterization, we can reduce the dimensional number of the unknown variable effectively. 

We plot the error functions $\|\mathcal{K}(\varphi)-b\|$ versus the accepted samples number in Fig. \ref{err1} corresponding to Fig. \ref{reconstruction1}-\ref{reconstruction3}. It can be seen that the errors decay at the beginning and subsequently fluctuate within a narrow range. This reflects that the Markov chains become stable after some iterations. 
 

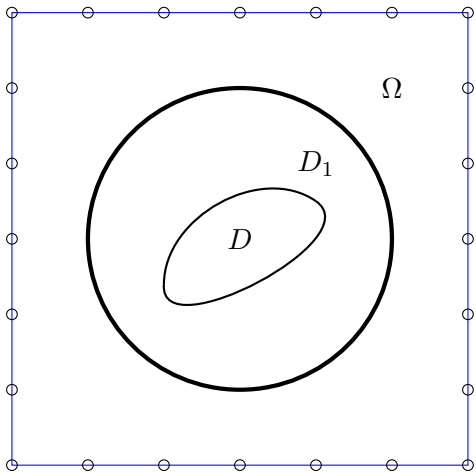
\begin{figure}\centering
 \begin{tikzpicture}
\draw [blue] (-2,-2) rectangle (4,4);
\draw [ultra thick] (1,1) circle [radius=2];
\draw[thick] (0,0.4) to [out=90,in=145] (2,1.5);
\draw[thick] (2,1.5) to [out=-35,in=-95] (0,0.4);
\node at (1,1) {$D$};
\node at (3, 3) {$\Omega$};
\node at (2, 2) {$D_1$};
\draw 
(-2,-2) circle (2pt) 
(-1,-2) circle (2pt) 
(0,-2) circle (2pt) 
(1,-2) circle (2pt) 
(2,-2) circle (2pt) 
(3,-2) circle (2pt) 
(4,-2) circle (2pt)    
(4,-1) circle (2pt)   (4,0) circle (2pt) (4,1) circle (2pt) (4,2) circle (2pt) 
(4,3) circle (2pt) (4,4) circle (2pt) (3,4) circle (2pt) (2,4) circle (2pt)
(1,4) circle (2pt)  (0,4) circle (2pt) (-1,4) circle (2pt) (-2,4) circle (2pt) 
(-2,3) circle (2pt) (-2,2) circle (2pt) (-2,1) circle (2pt) (-2,0) circle (2pt) 
(-2,-1) circle (2pt)
 node[align=right,  below] {};
\end{tikzpicture}
\caption{The problem geometry.}\label{fig:digit}
\end{figure}

\begin{figure}\centering
\includegraphics[width=5cm]{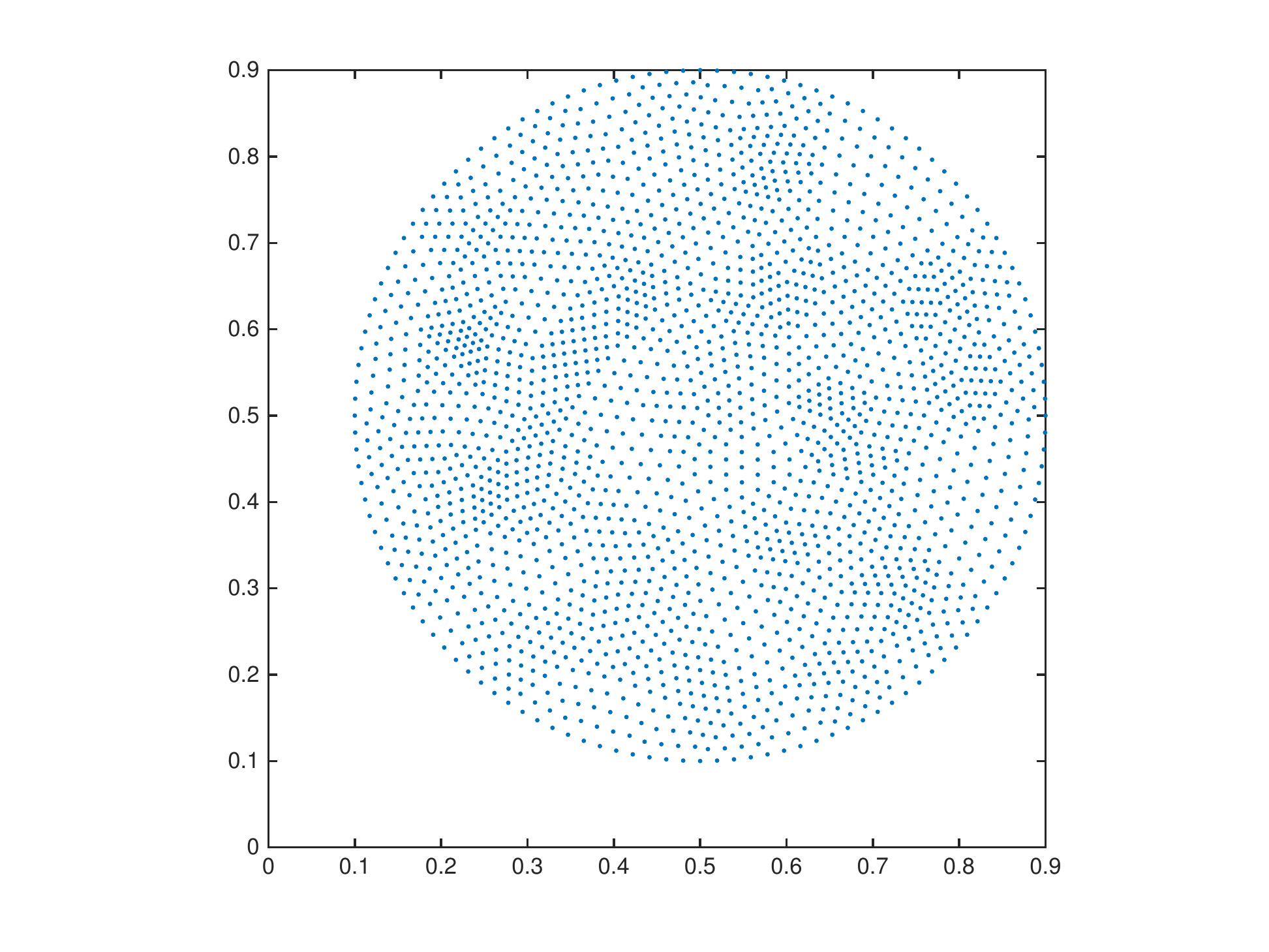}  
\includegraphics[width=5cm]{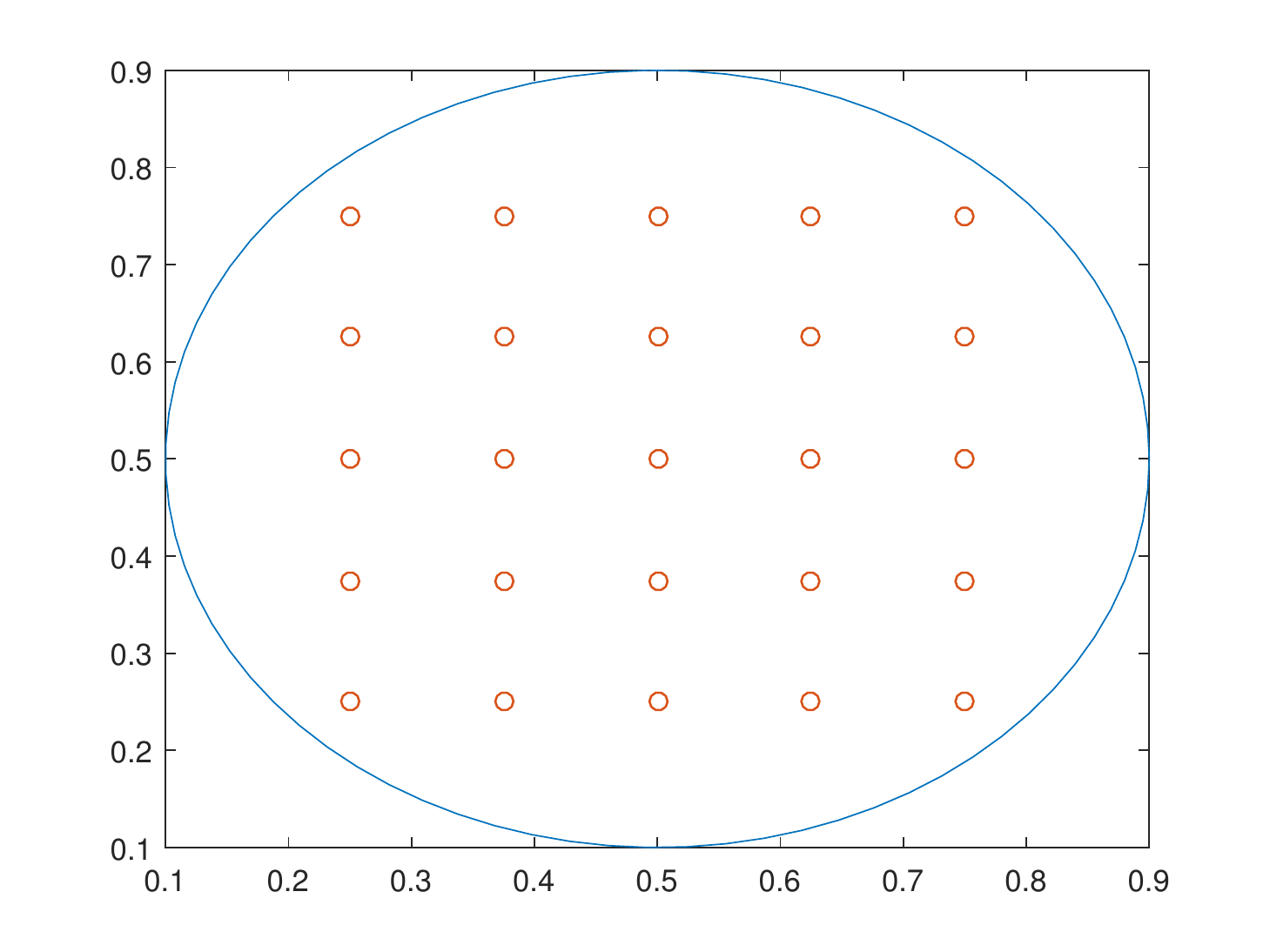}  
\caption{The centers of the RBF function, left: 2161 center points; right: 25 center points.}  
\label{centersrbf} 
\end{figure}

\begin{figure}\centering
\includegraphics[width=3.5cm]{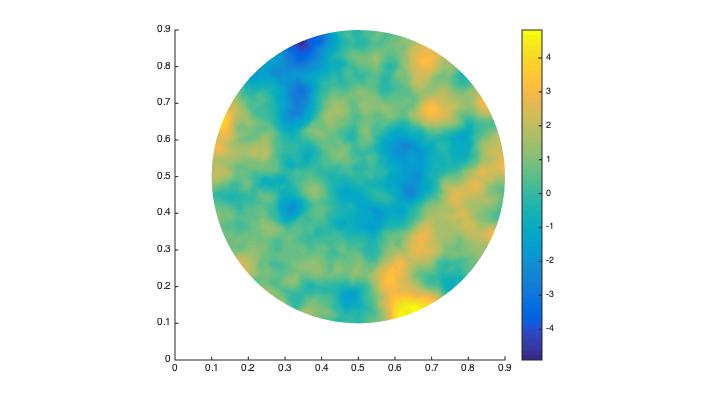}  
\includegraphics[width=3.5cm]{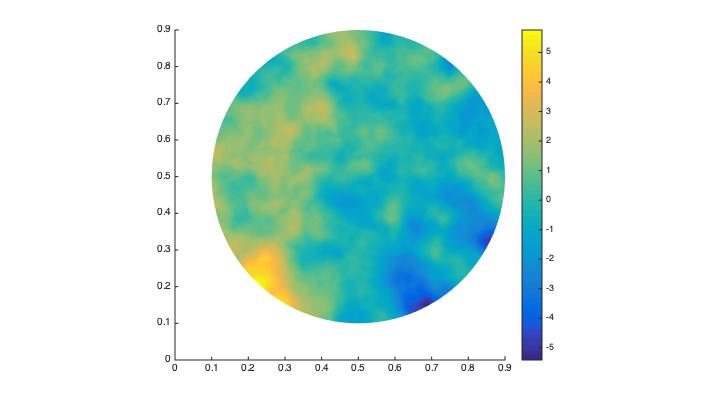}  
\includegraphics[width=3.5cm]{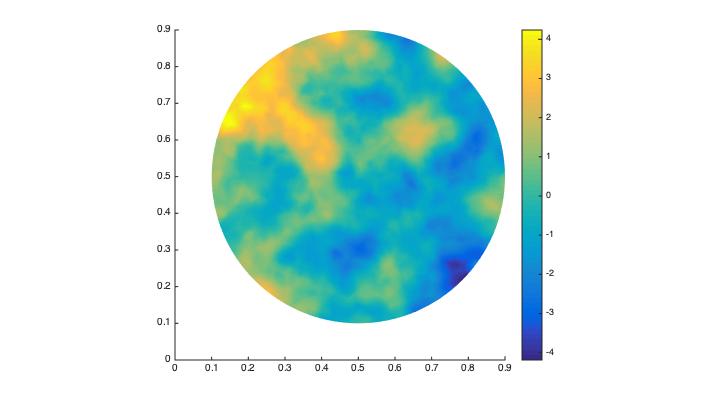}\\
\includegraphics[width=3.5cm]{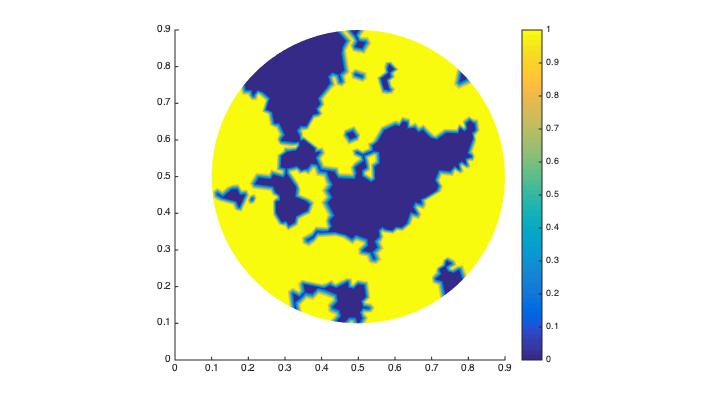}  
\includegraphics[width=3.5cm]{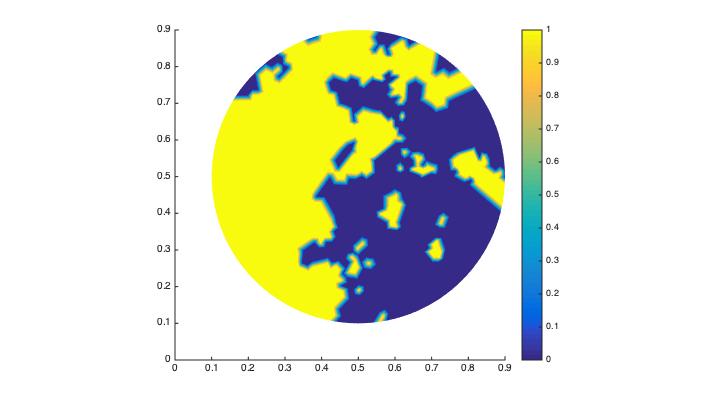}  
\includegraphics[width=3.5cm]{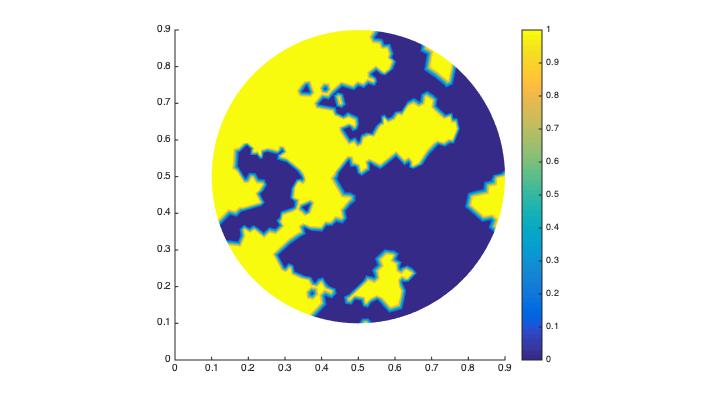}\\
\includegraphics[width=3.5cm]{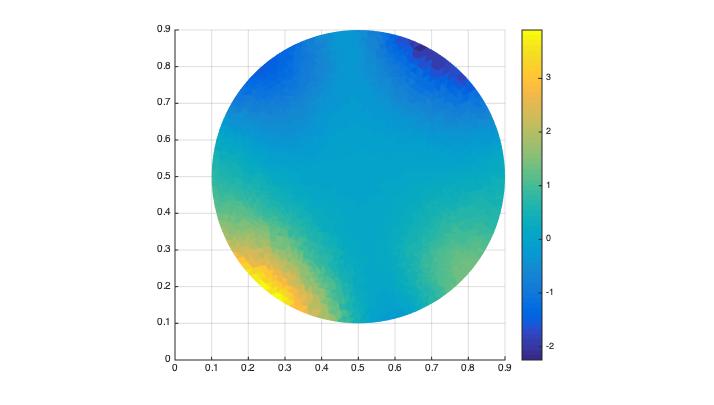}  
\includegraphics[width=3.5cm]{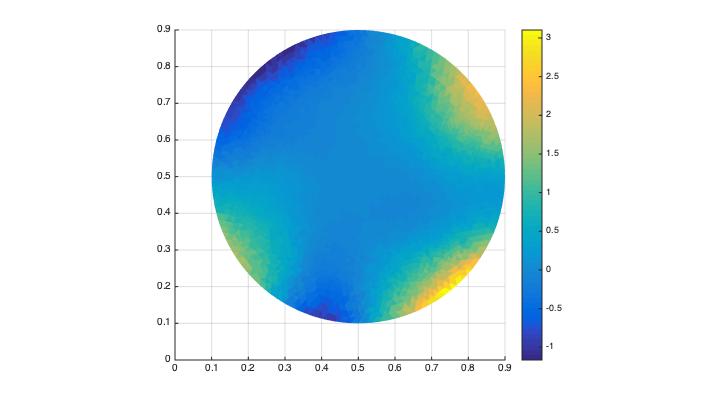}  
\includegraphics[width=3.5cm]{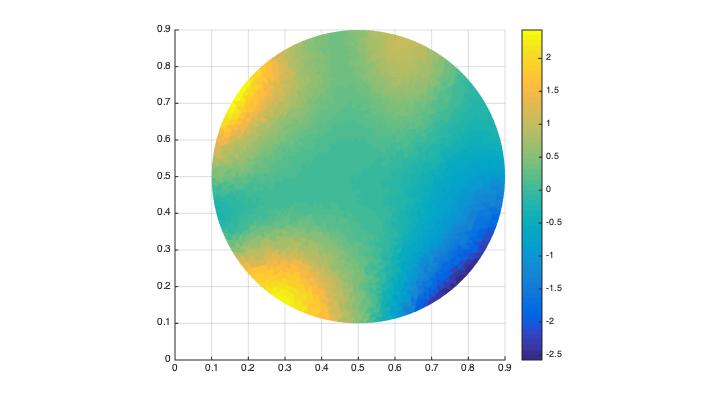}  \\
\includegraphics[width=3.5cm]{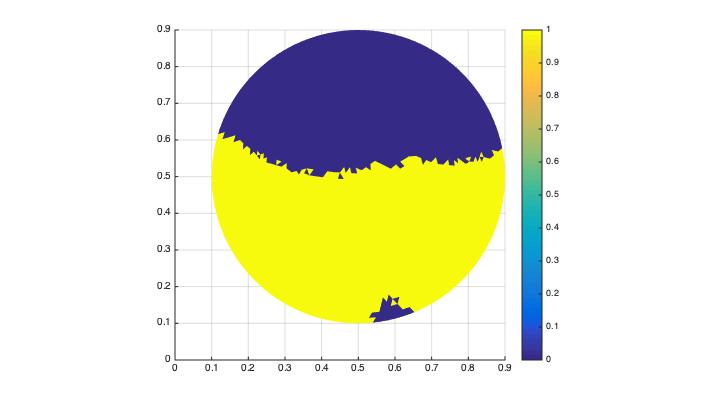}  
\includegraphics[width=3.5cm]{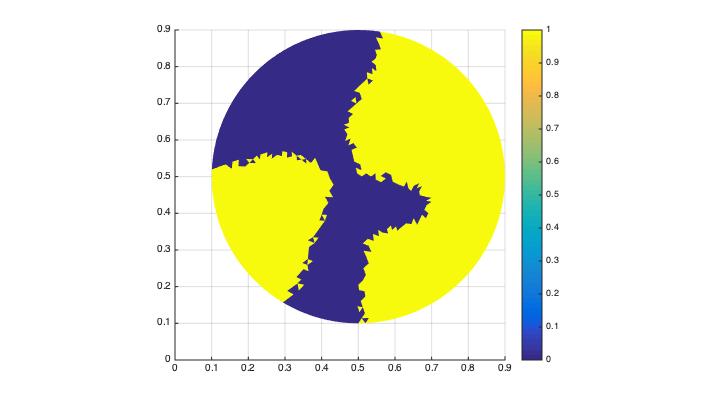}  
\includegraphics[width=3.5cm]{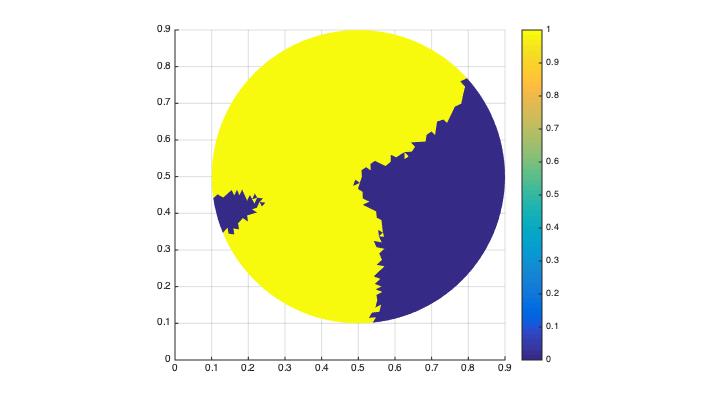}
\caption{Lines 1, 2: Whittle Mat\'ern random field prior samples and the corresponding $f$; Lines 3, 4: RBF expansion samples and the corresponding $f$.}  
\label{prior1} 
\end{figure}


\begin{figure}\centering
\subfigure[Lshape]{
\includegraphics[width=3.5cm]{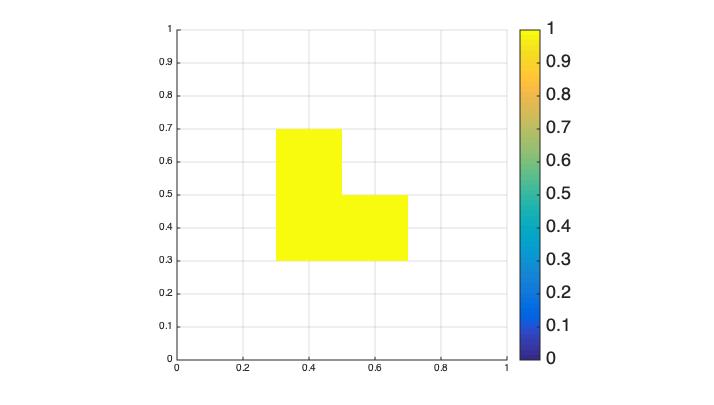} 
\includegraphics[width=3.5cm]{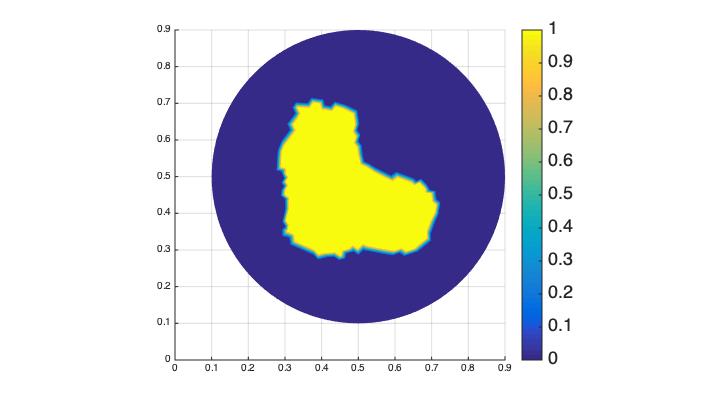}  
\includegraphics[width=3.5cm]{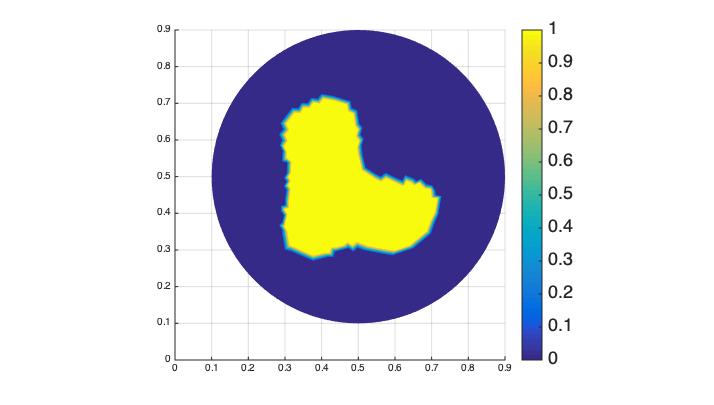}  }
\subfigure[Single disk]{
\includegraphics[width=3.5cm]{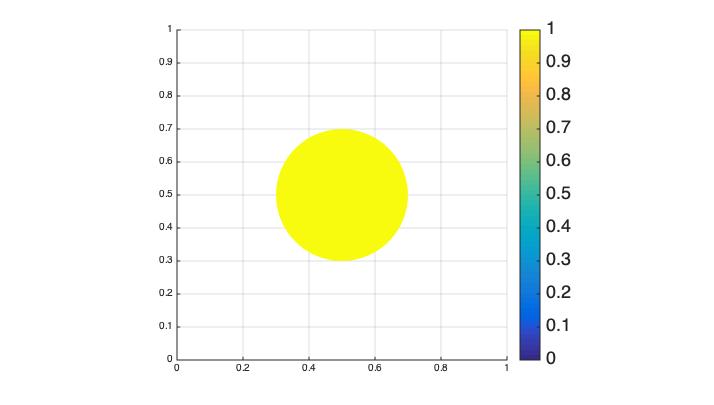} 
\includegraphics[width=3.5cm]{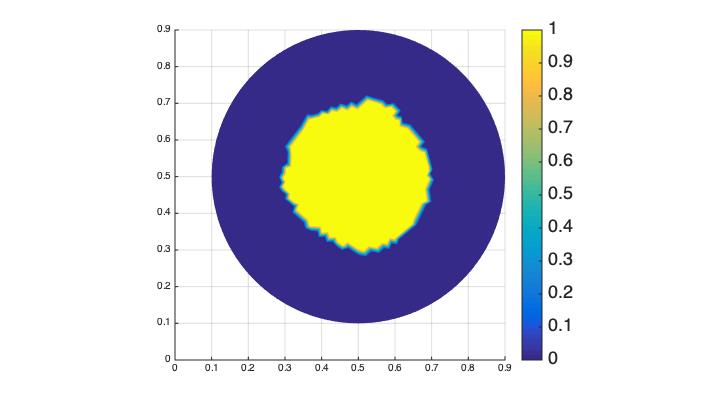}  
\includegraphics[width=3.5cm]{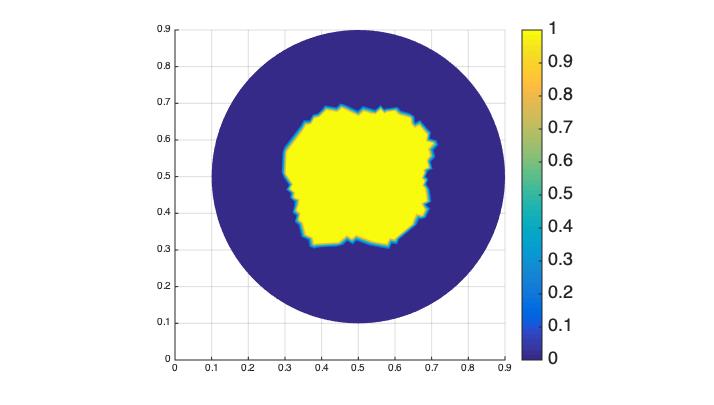} }
\subfigure[Two disks with the same medium value]{
\includegraphics[width=3.5cm]{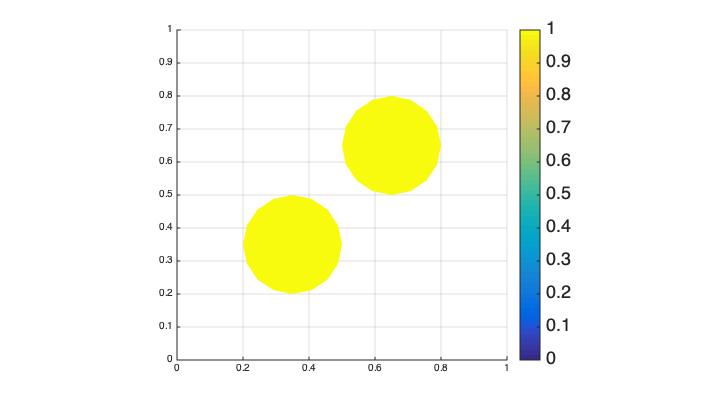} 
\includegraphics[width=3.5cm]{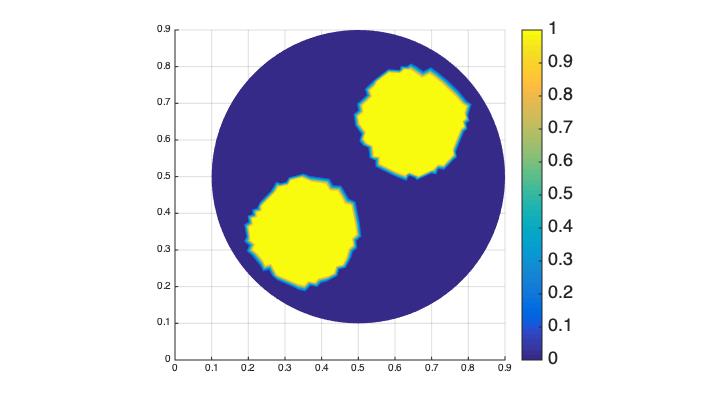}  
\includegraphics[width=3.5cm]{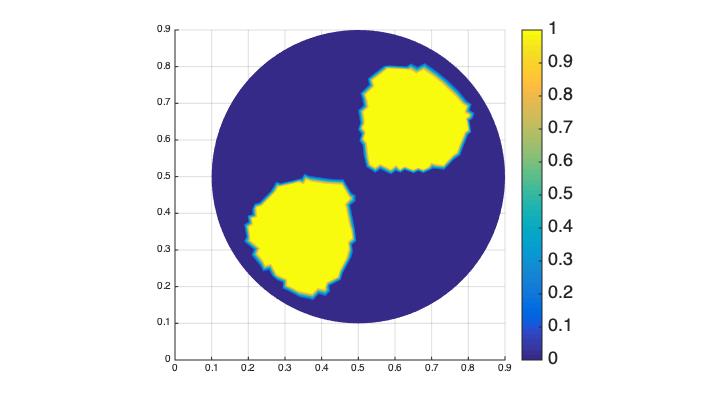}  }
\caption{The numerical reconstructions for $\delta=0.01$. Left: the true domain; Middle: the reconstructions using Mat\'ern field prior; Right: the reconstructions using the RBF expansion.}  
\label{reconstruction1} 
\end{figure}


\begin{figure}\centering
\subfigure[Two disks with different medium values]{
\centering
\includegraphics[width=3.5cm]{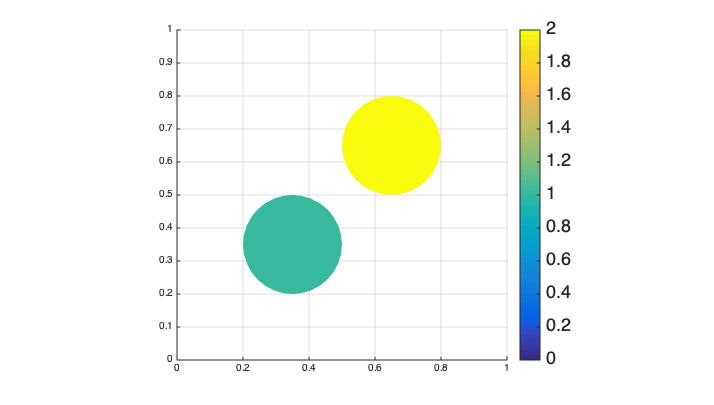} 
\includegraphics[width=3.5cm]{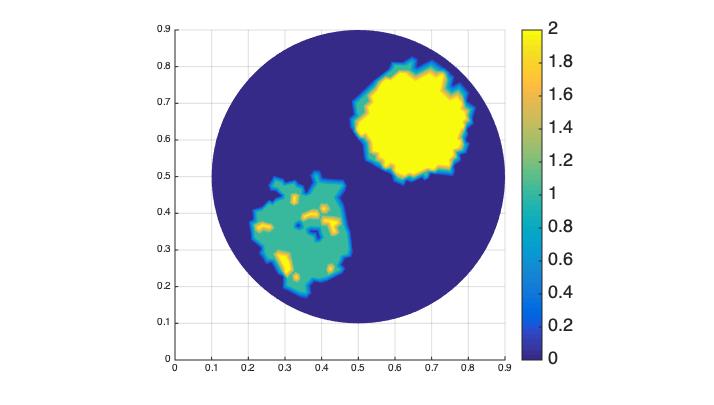}  
\includegraphics[width=3.5cm]{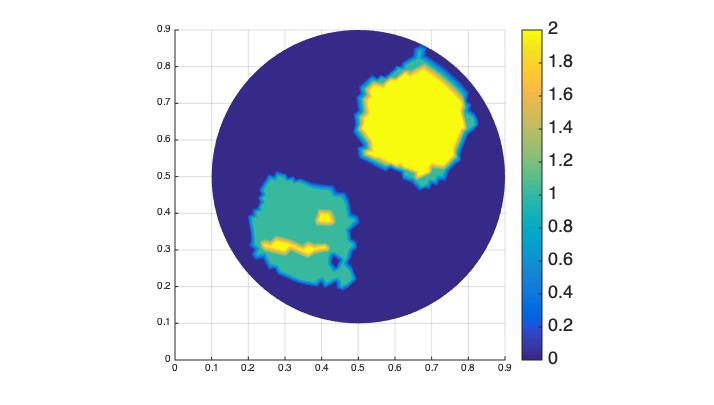} }
\subfigure[Two different domains with different medium values]{
\centering
\includegraphics[width=3.5cm]{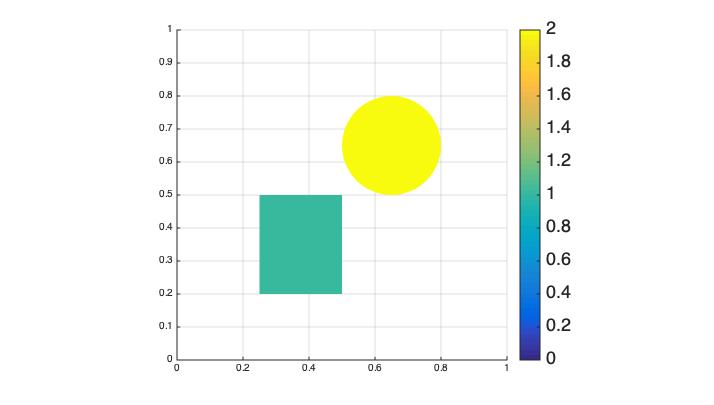} 
\includegraphics[width=3.5cm]{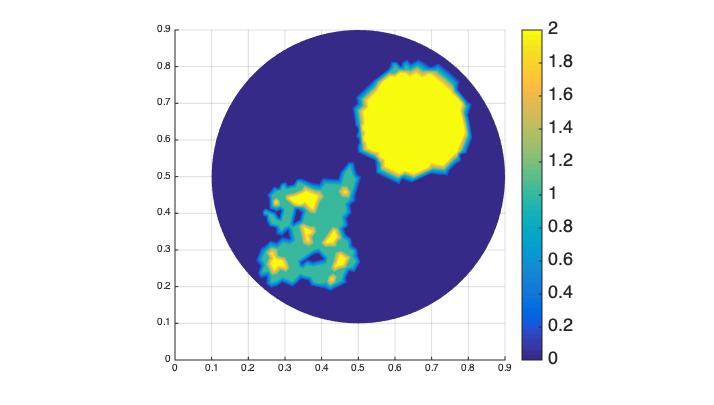}  
\includegraphics[width=3.5cm]{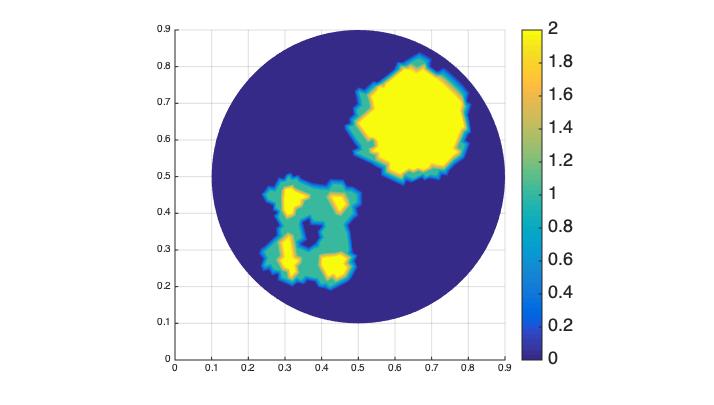} } 
\subfigure[Embedded domain with different medium values]{
\centering
\includegraphics[width=3.5cm]{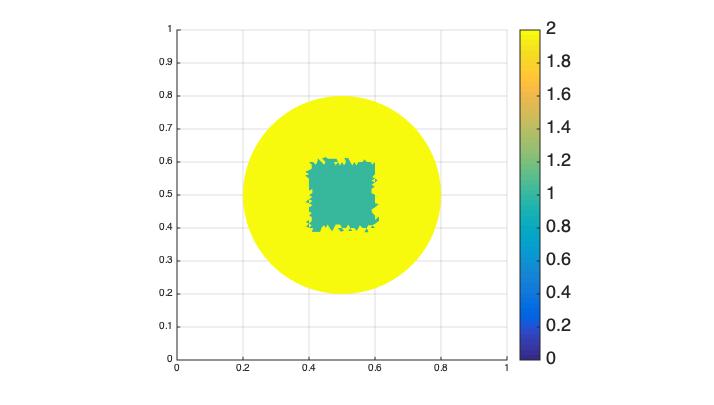} 
\includegraphics[width=3.5cm]{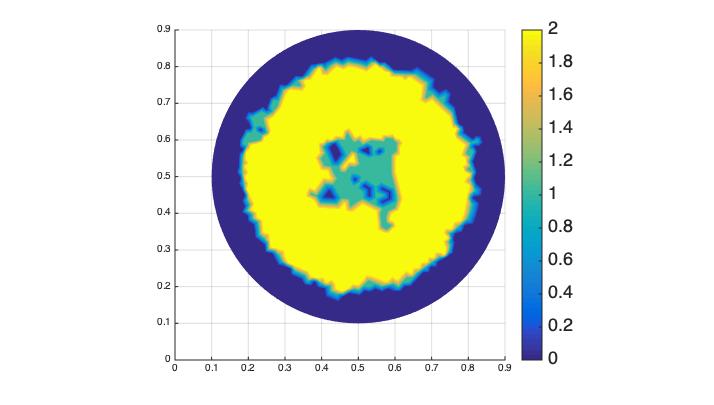}  
\includegraphics[width=3.5cm]{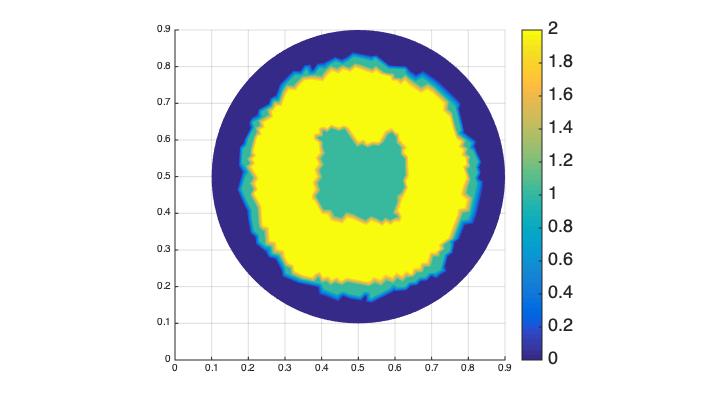}  }
\caption{The numerical reconstructions with $\delta=0.01$ for two kinds of mediums.  Left: the true domain; Middle: the reconstructions using Mat\'ern field prior; Right: the reconstructions using the RBF expansion.}  
\label{reconstruction2} 
\end{figure}




\begin{figure}\centering
\subfigure[Lshape]{
\includegraphics[width=3.2cm]{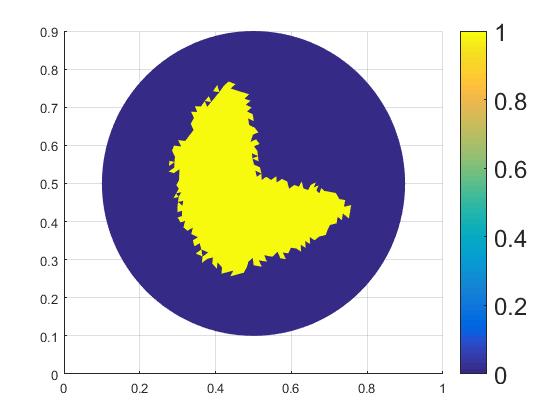} }
\subfigure[Single disk]{
\includegraphics[width=3.2cm]{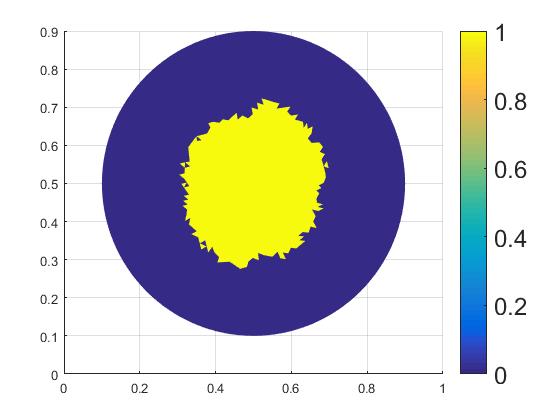} }
\subfigure[Two disks with the same medium value]{
\includegraphics[width=3.2cm]{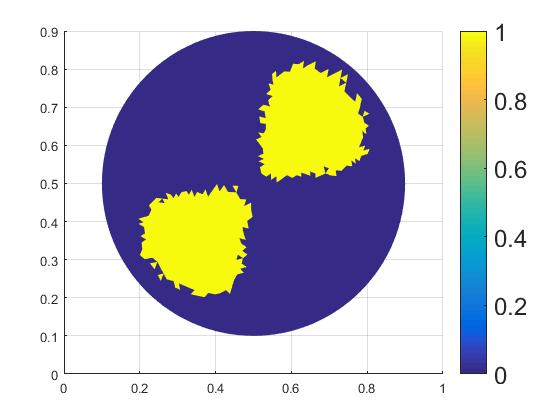} }
\caption{The numerical reconstructions  with $\delta=0.02$ using $25$ centers in the RBF expansion for single medium case.}  
\label{reconstruction3} 
\end{figure}

\begin{figure}\centering
\subfigure[Corresponding to Fig. \ref{reconstruction1}]
{\includegraphics[width=5cm]{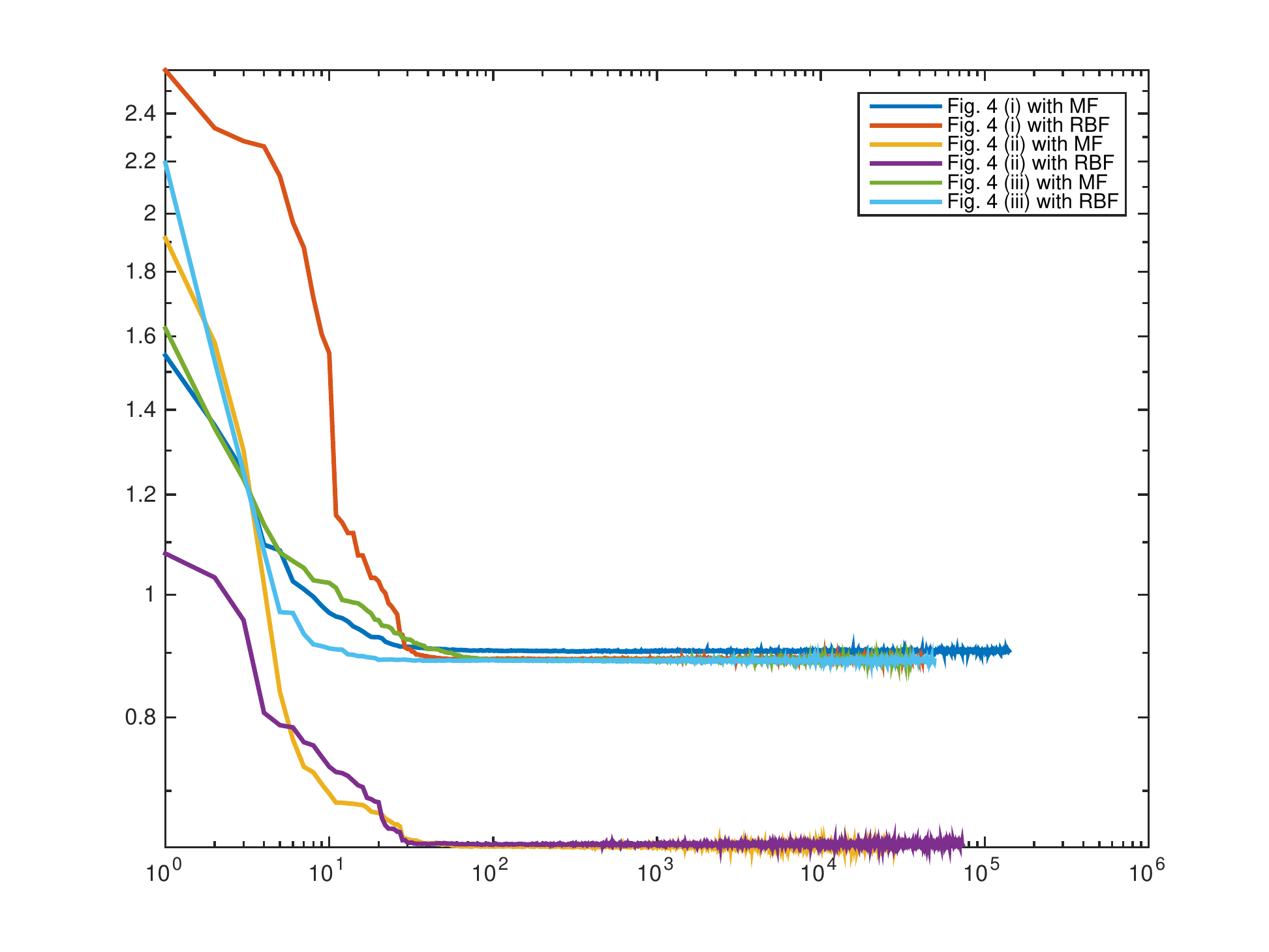} }
\subfigure[Corresponding to Fig. \ref{reconstruction2}]
{\includegraphics[width=5cm]{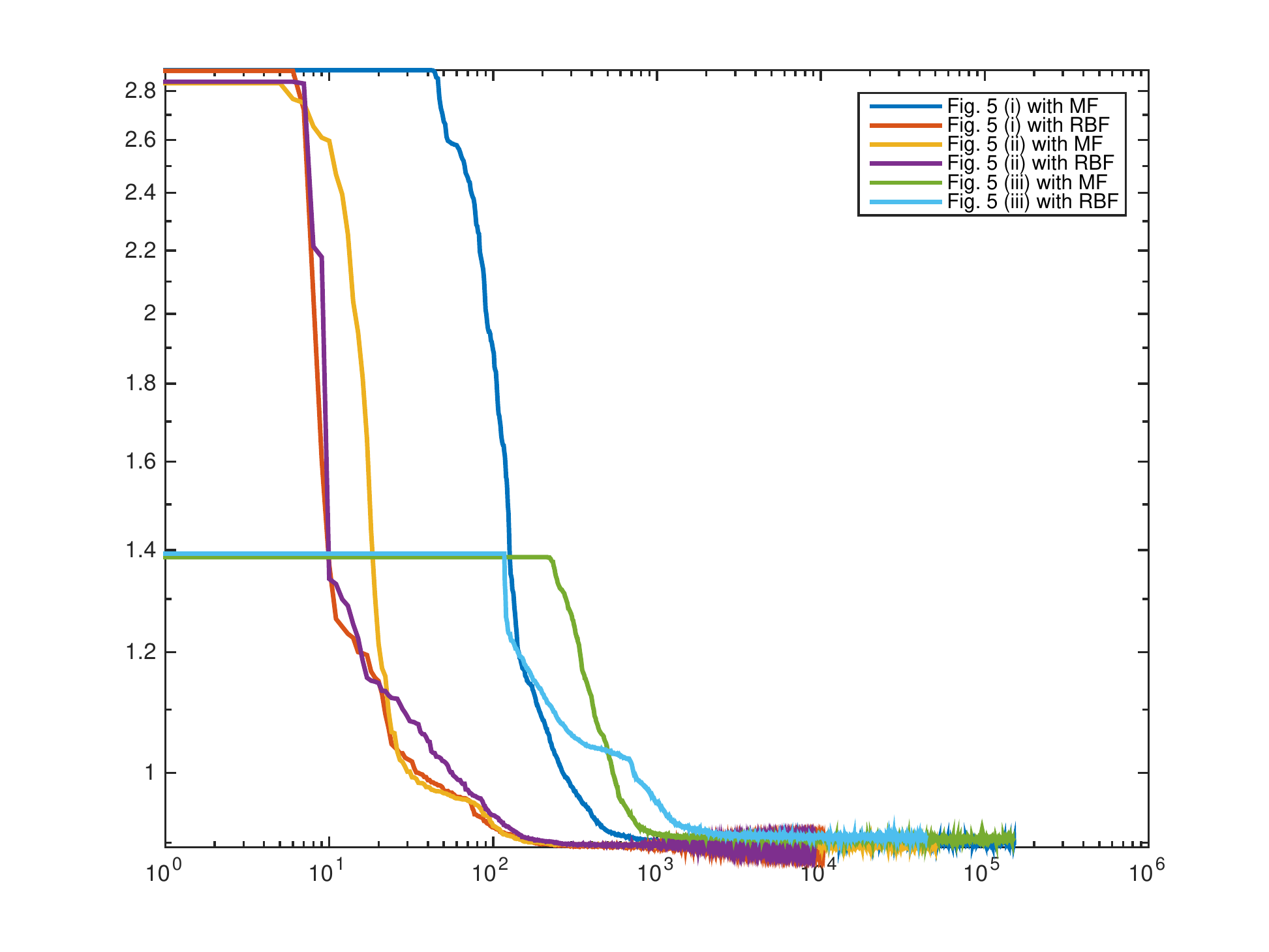} }
\subfigure[Corresponding to Fig. \ref{reconstruction3}]
{\includegraphics[width=5cm]{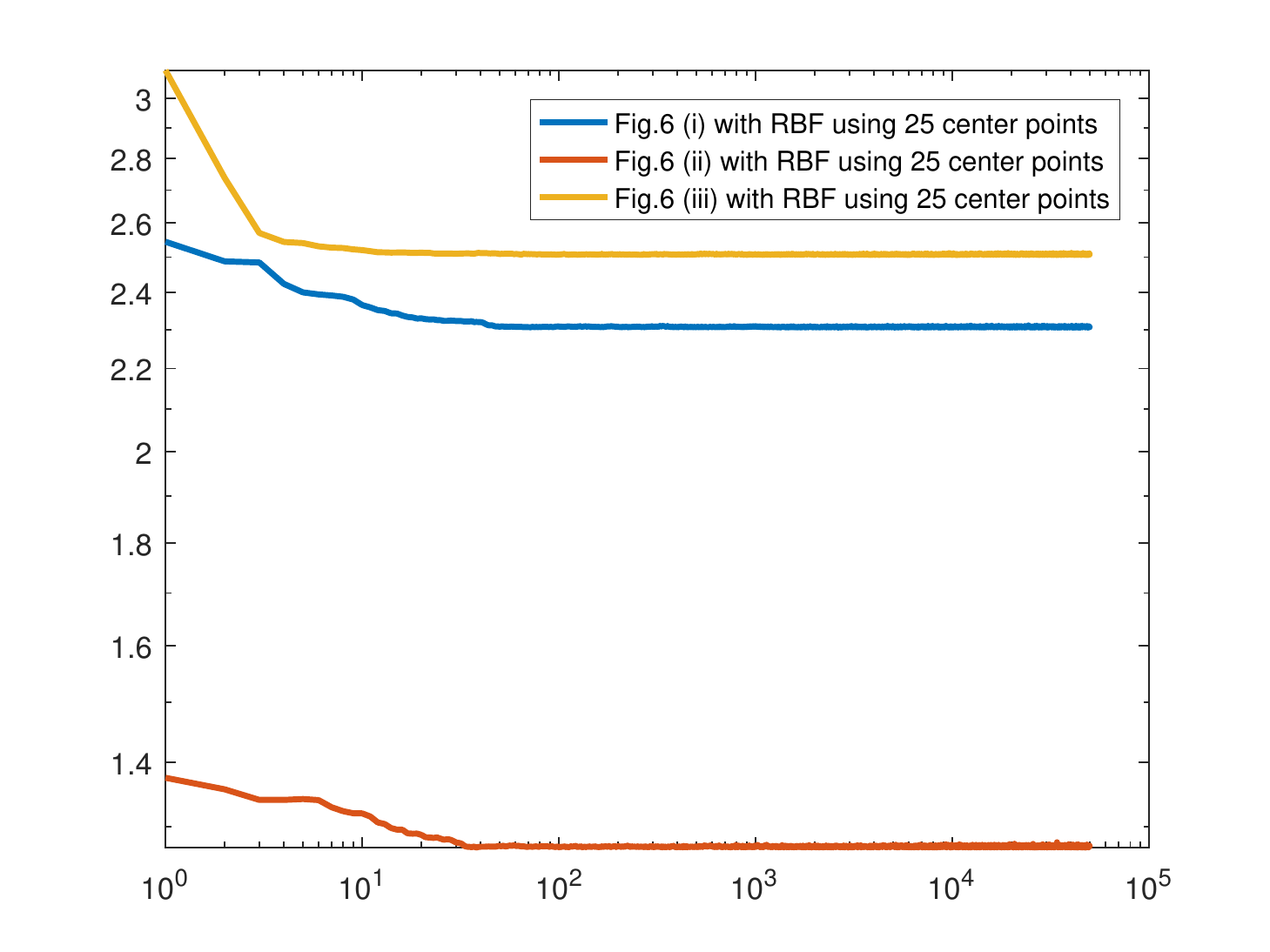} }
\caption{The corresponding model data misfit functions $\|\mathcal{K}(\varphi)-b\|$ for the accepted samples.}  
\label{err1} 
\end{figure}

\end{document}